\documentclass[12pt]{article}

\usepackage{amsmath,mathtools,amssymb,amsfonts,amsthm,graphicx,color,verbatim}

\definecolor{Blue}{rgb}{0.3,0.3,0.9}

 \usepackage{graphicx}
\usepackage{bbm}

\usepackage{amsmath, amsthm, amssymb}
  \theoremstyle{plain}
      \newtheorem{theorem}{Theorem}[section]

      \newtheorem{Proposition}[theorem]{Proposition}
      
      \theoremstyle{definition}

      \theoremstyle{remark}

\numberwithin{equation}{section}

\title{Population Processes with Immigration}
\author{Dan Han
 \\\small Department of Mathematics and Statistics\\[-0.8ex]
\small University of North Carolina at Charlotte, Charlotte, NC 28223,USA,\\\\
Stanislav Molchanov
 \\\small Department of Mathematics and Statistics\\[-0.8ex]
\small University of North Carolina at Charlotte, Charlotte, NC 28223,USA,\\
\small National Research University, Higher School of Economics, Russian Federation\\\\
Joseph Whitmeyer
 \\\small Department of Sociology\\[-0.8ex]
\small University of North Carolina at Charlotte, Charlotte, NC 28223,USA\\}

\begin{document}

\maketitle

\begin{abstract}
The paper contains the complete analysis of the Galton-Watson models with immigration, including the processes in the random environment, stationary or non-stationary ones. We also study the branching random walk on $Z^{d}$ with immigration and prove the existence of the limits for the first two correlation functions.

\end{abstract}


\section{Introduction}
A problem with many single population models of population dynamics involving processes of birth, death, and migration is that the populations do not attain steady states or do so only under critical conditions.  One solution is to allow immigration, which can stabilize the population when the birth rate is less than the mortality rate.

Here, we present analysis of several models that incorporate immigration.  The first two are spatial Galton-Watson processes, the first with no migration and the second with finite Markov chain spatial dynamics (see section 2 and 3 respectively).  The third model allows migration on $\mathbb{Z}^d$ (see section 4).  The remaining models all involve random environments in some way (see section 5).  Two are again Galton-Watson processes, the first with a random environment based on population size and the second with a random environment given by a Markov chain. The last two models have birth, death, immigration, and migration in a random environment allowing in some way non-stationarity in both space and time.  We study in this paper only first and second moments. We will return to the complete analysis of the models with immigration in another publication. It will include a theorem about the existence of steady states and an analysis of the stability of these states.


\section{Spatial Galton-Watson process with immigration. \- No migration and no random environment.}
\subsection{Moments}
Assume that at each site for each particle we have birth of one new particle with rate $\beta$ and death of the particle with rate $\mu$.  Also assume that regardless of the number of particles at the site we have immigration of one new particle with rate $k$ (this is a simplified version of the process in \cite{sev57}).  Assume that $\beta < \mu$, for otherwise the population will grow exponentially.  Assume we start with one particle at each site.  In continuous time, for a given site $x$, $x \in Z^d$, we can obtain all moments recursively by means of the Laplace transform with respect to $n(t,x)$, where $n(t,x)$ is the population size at time $t$ at $x$
\[\varphi_t(\lambda) = E\, e^{-\lambda n(t,x)} = \sum_{j=0}^\infty P\{n(t,x) = j\} e^{-\lambda j}. \]
Specifically, for the $j$th moment, $m_j$ 
\begin{equation}\label{LTmom}
m_j(t,x) = (-1)^j \frac{\partial^j \varphi}{\partial \lambda^j}|_{\lambda = 0}.
\end{equation}

A partial differential equation for $\varphi_t(\lambda)$ can be derived using the forward Kolmogorov equations
\begin{equation}\label{basic} 
n(t+dt,x)= n(t,x)+ \xi_{dt}(t,x)
\end{equation}
where the r.v.\! $\xi$ is defined

        \begin{equation}\label{imm1}
        \xi_{dt}(t,x) = \left\{
        \begin{array}{l}
        +1 \hspace {.5 cm} \beta n(t,x)dt + k dt \\
        -1 \hspace {.5 cm} \mu n(t,x)dt \\
       
	0 \hspace {.8 cm} 1 - ((\beta + \mu) n(t,x) + k) dt \\
        \end{array}
        \right.
        \end{equation}  
In other words, our site ($x$) in a small time interval ($dt$) can gain a new particle at rate $\beta$ for every particle at the site or through immigration with rate $k$; it can lose a particle at rate $\mu$ for every particle at the site; or no change at all can happen.
Because our model is homogeneous in space, we can write $n(t)$ for $n(t,x)$.
This leads to the general differential equation
	\begin{eqnarray*}
	&&\frac{\partial \varphi_t(\lambda)}{\partial t} = \varphi_t(\lambda) \left(\mu n(t)e^\lambda - ((\beta + \mu) n(t) + k) + (\beta n(t) + k)e^{-\lambda}\right) \\
	&&\varphi_0(\lambda) = e^{-\lambda}
	\end{eqnarray*}
from which we can calculate the recursive set of differential equations
	\begin{eqnarray*}
	&&\frac{\partial \varphi_t(\lambda)^{(j)}}{\partial t} = \varphi_t(\lambda)^{(j)} \left(\mu n(t)e^\lambda - ((\beta + \mu) n(t) + k) + (\beta n(t) + k)e^{-\lambda}\right) + \\
	&&\hspace{2 cm} + \sum_{i=1}^j \binom{j}{i} \varphi_t(\lambda)^{(j-i)}\left(\mu n(t) e^\lambda + (-1)^i(\beta n(t) + k)e^{-\lambda}\right) \\
	&&\varphi_0(\lambda)^{(j)} = (-1)^{j} e^{-\lambda}
	\end{eqnarray*}
Applying \ref{LTmom} we obtain a set of recursive differential equations for the moments
	\begin{eqnarray}\label{GW4}
	&&\frac{d m_j(t)}{d t} = \sum_{i=1}^j \binom{j}{i} \left((\beta + (-1)^i \mu) m_{j-i+1} +  m_{j-i}\right) \nonumber\\
	&&\hspace{1.4 cm} = j(\beta - \mu) m_{j} + s_j \\
	&&m_j(0) = 1 \nonumber
	\end{eqnarray}
where $s_j$ denotes a linear expression involving lower order moments and where we define $m_0 = 1$.
For example, the differential equations for the first and second moments are
	\begin{eqnarray*}
	&&\frac{d m_1(t)}{d t} = (\beta - \mu) m_1(t) + k \\
	&&m_1(0) = 1
	\end{eqnarray*}
and
	\begin{eqnarray*}
	&&\frac{d m_2(t)}{d t} = 2(\beta - \mu)m_2(t) + (\beta + \mu + 2k) m_1(t) + k  \\
	&&m_2(0) = 1
	\end{eqnarray*}

These have the solutions:

\[m_1(t) = \frac{k}{\mu-\beta} + (1 - \frac{k}{\mu-\beta})e^{-(\mu-\beta)t} \]
and
\begin{equation*}
\begin{split}
m_2(t) =& \frac{k(k + \mu)}{(\mu-\beta)^2} + \frac{\mu^2 - 2k^2 -\beta^2 + k \mu - 3k\beta}{(\mu-\beta)^2} e^{-(\mu-\beta)t} + \\
&+ \frac{k^2 + 2\beta^2 +3k \beta -2 \mu \beta -2 k \mu}{(\mu-\beta)^2} e^{-2(\mu-\beta)t} \\
\end{split}
\end{equation*}

Again, given that we have assumed that $\mu > \beta$, in other words, the birth rate is not high enough to maintain the population size, as $t \to \infty$
\[m_1(t) \xrightarrow[t \to \infty]{} \frac{k}{\mu-\beta} \]

\[m_2(t) \xrightarrow[t \to \infty]{} \frac{k(k + \mu)}{(\mu-\beta)^2} \]
and 
\[\mathrm{Var}(n(t)) = m_2(t) - m_1^2(t) \xrightarrow[t \to \infty]{} \frac{\mu k}{(\mu-\beta)^2}. \]
Moreover, it is clear from Eq. \ref{GW4} that all the moments are finite.

In other words, the population size will approach a finite limit, which can be regulated by controlling the immigration rate $k$, and this population size will be stable, as indicated by the fact that the limiting variance is finite. Without immigration, i.e., if $k = 0$, the population size will decay exponentially.  Another possibility, because all sites are independent and there are no spatial dynamics, is for there to be immigration at some sites, which therefore reach stable population levels, and not at others, where the population thus decreases exponentially.  Of course, if the birth rate exceeds the death rate, $\beta > \mu$, $m_1(t)$ increases exponentially and immigration has negligible effect, as shown by the solution for $m_1(t)$.


\subsection{Local CLT}
Setting $\lambda_n=n\beta+k$, $\mu_n=n\mu$, we see that the model given by Eqs. \ref{basic} and \ref{imm1} is a particular case of the general random walk on $Z_+^1=\lbrace{0,1,2,\cdots}\rbrace$ with generator 
\begin{align}
\mathcal{L}\psi(n)&=\psi(n+1)\lambda_{n}-(\lambda_{n}+\mu_{n})\psi(n)+\mu_{n}\psi(n-1), \quad n\geqslant 0 \\
\mathcal{L}\psi(0)&=k\psi(1)-k\psi(0)
\end{align}

The theory of such chains has interesting connections to the theory of orthogonal polynomials, the moments problem, and related topics (see \cite{km}).  We recall several facts of this theory.

\begin{enumerate}
\item[a.] Equation $\mathcal{L}\psi=0, x\geqslant 1$, (i.e., the equation for harmonic functions) has two linearly independent solutions:
\begin{equation}
\begin{array}{ll}
\psi_{1}(n) &\equiv 1 \\
\psi_{2}(n) &= \left\{ 
\begin{array}{ll}
0 & n=0 \\
1 & n=1 \\
1+\frac{\mu_1}{\lambda_1}+\frac{\mu_1\mu_2}{\lambda_1\lambda_2}+\cdots+\frac{\mu_1\mu_2\cdots \mu_{n-1}}{\lambda_1\lambda_2\cdots \lambda_{n-1}} & n\geqslant2\\
\end{array} \right. \\
\end{array} 
\end{equation}

\item[b.] Denoting the adjoint of $\mathcal{L}$ by $\mathcal{L}^{*}$, equation $\mathcal{L}^{*}\pi=0$ (i.e., the equation for the stationary distribution, which can be infinite) has the positive solution 

\begin{align}
\pi(1)&=\frac{\lambda_{0}}{\mu_{1}} \pi(0)\\
\pi(2)&=\frac{\lambda_{0} \lambda_{1}}{\mu_{1} \mu_{2}} \pi(0)\\
\cdots\\
\pi(n)&=\frac{\lambda_{0} \lambda_{1} \cdots\lambda_{n-1}}{\mu_{1} \mu_{2} \cdots \mu_{n}} \pi(0)
\end{align}

\end{enumerate}

This random walk is ergodic (i.e., $n(t)$ converges to a statistical equilibrium, a steady state) if and only if the series 
$1+\frac{\lambda_0}{\mu_1}\cdots+\frac{\lambda_0\lambda_1}{\mu_1\mu_2}+\cdots
+\frac{\lambda_0\lambda_1\cdots\lambda_{n-1}}{\mu_1\mu_2\cdots\mu_n}$ converges.  In our case, $$x_n=\frac{\lambda_0\cdots\lambda_{n-1}}{•\mu_1\cdots\mu_n}=\frac{k(k+\beta)\cdots(k+(\mu-1))\beta}{\mu (2\mu) \cdots (n\mu)}.$$
If $\beta>\mu$, then, for $n>n_0$, for some fixed $\varepsilon>0$, $\frac{k+(n-1)\beta}{n\mu}>1+\varepsilon$, that is, $x_n \geq C^n$, for $C>1$ and $n \geq n_1(\varepsilon)$, and so $\sum x_n=\infty$.
In contrast, if $\beta<\mu$, then, for some $0<\varepsilon<1$, $\frac{k+(n-1)\beta}{n\mu}<1-\varepsilon$, and $x_n\leq q^n$, for $0<q<1$ and $n>n_1(\varepsilon)$; thus, $\sum x_n <\infty$.  In this ergodic case, the invariant distribution of the random walk $n(t)$ is given by the formula $$\pi(n)=\frac{1}{\tilde{S}}\frac{\lambda_0\cdots\lambda_{n-1}}{•\mu_1\cdots\mu_n},$$ where $$\displaystyle\tilde{S}=1+\frac{k}{\mu}+\frac{k(\beta+k)}{\mu (2\mu)}+\cdots+\frac{k(k+\beta)\cdots(\beta(n-1)+k)}{\mu (2\mu) \cdots (n\mu)}+\cdots.$$

\begin{theorem}[Local Central Limit theorem]
Let $\beta<\mu$. If $l=O(k^{2/3})$, then, for the invariant distribution $\pi(n)$
\begin{align}
\pi(n_{0}+l)\sim\frac{e^{-\frac{l^2}{2\sigma^2}}}{\sqrt{2\pi \sigma^2}}\,\,\,\, as\,\,\, k \rightarrow \infty
\end{align}
where $\sigma^2=\frac{\mu k}{(\mu-\beta)^2}$, $n_{0}\sim \frac{k}{\mu-\beta}$.
\end{theorem}
\begin{proof}
\begin{align*}
\pi(n)&=\frac{1}{\tilde{S}}\frac{k(k+\beta)\cdots(k+\beta(n-1))}{\mu (2\mu) \cdots (n\mu)}\\
&=\frac{1}{\tilde{S}}\left(\frac{\beta}{\mu}\right)^{n}\frac{\frac{k}{\beta}(\frac{k}{\beta}+1)\cdots(\frac{k}{\beta}+n-1)}{n\,!}\\
&=\frac{1}{\tilde{S}}\left(\frac{\beta}{\mu}\right)^{n}\frac{\Gamma(\frac{k}{\beta}+n)}{\Gamma(\frac{k}{\beta})n\,!}.
\end{align*}
We see $\tilde{S}$ is a degenerate hypergeometric series, thus 
\begin{align*}
\tilde{S}=\left(1-\frac{\beta}{\mu}\right)^{-\frac{k}{\beta}}.
\end{align*}
Set 
\begin{equation}\label{defa}
a_{n}=\displaystyle\left(\frac{\beta}{\mu}\right)^{n}\frac{\Gamma(\frac{k}{\beta}+n)}{\Gamma(\frac{k}{\beta})n\,!}.
\end{equation}
Then, $\pi(n)=\displaystyle\frac{a_{n}}{\tilde{S}}$.
We have
\begin{align*}
a_{n+l}&=a_{n}\left(\frac{\beta}{\mu}\right)^l(1+\frac{\frac{k}{\beta}-1}{n+1})(1+\frac{\frac{k}{\beta}-1}{n+2})\cdots(1+\frac{\frac{k}{\beta}-1}{n+l})\\
&=a_{n}\prod_{i=1}^{l}\frac{\beta}{\mu}(1+\frac{\frac{k}{\beta}-1}{n+i})\\
&=a_{n}\prod_{i=1}^{l}\frac{1+\frac{\beta(i-1)(\mu-\beta)}{\mu k}}{1+\frac{i(\mu-\beta)}{k}}\\
&=a_n\prod_{i=1}^l\frac{\frac{\beta}{\mu}(n+i-1)+\frac{k}{\mu}}{n+i}
\end{align*}
and because $a_{n_0}\sim\frac{k}{\mu-\beta}$
\begin{align*}
a_{n_0+l}&\sim a_{n_0}\prod_{i=1}^l\frac{\frac{\beta}{\mu}(\frac{k}{\mu-\beta}+i-1)+\frac{k}{\mu}}{\frac{k}{\mu-\beta}+i}\\
&=a_{n_0}\prod_{i=1}^l\frac{\beta(i-1)(\mu-\beta)+k\mu}{i(\mu-\beta)\mu+k\mu}\\
&=a_{n_0}\prod_{i=1}^l\frac{1+\frac{\beta(i-1)(\mu-\beta)}{k\mu}}{1+\frac{i(\mu-\beta}{k}}\\
&=a_{n_0}e^{\sum\limits_{i=1}^{l}[\ln(1+\frac{\beta(i-1)(\mu-\beta)}{\mu k}) - \ln(1+\frac{i(\mu-\beta)}{k})]}
\end{align*}

We consider $\sum\limits_{i=1}^{l}\ln(1+\frac{\beta(i-1)(\mu-\beta)}{\mu k})=\int_1^l\ln(1+\frac{(x-1)(\mu-\beta)\beta}{\mu k})dx+O(\ln(1+\frac{(l-1)(\mu-\beta)\beta}{\mu}))$ and \\
$\sum\limits_{i=1}^l\ln(1+\frac{i(\mu-\beta)}{k})dx=\int_1^l\ln(1+\frac{x(\mu-\beta)}{k})dx+O(\ln(1+\frac{l(\mu-\beta)}{k}))$\\
We integrate the series $\ln(1+x)=x-\frac{1}{2}x^2+\frac{1}{3}x^3-\cdots$, and take $l=O(k^{2/3})$

\begin{align*}
&\int_1^l \ln \left(1+\frac{(x-1)(\mu-\beta)\beta}{\mu k}\right)dx\\
&=\int_1^l \frac{(x-1)(\mu-\beta)\beta}{\mu k}dx-\frac{1}{2} \int_1^l \left(\frac{(x-1)(\mu-\beta)\beta}{\mu k}\right)^2dx+\cdots\\
&=\frac{(\mu-\beta)\beta}{\mu k} \left(\frac{l^2}{2}-l\right)-\frac{(u-\beta)^2\beta^2}{6\mu^2k^2}(l-1)^3+\cdots\\
&=\frac{(\mu-\beta)\beta}{\mu k}l^2+O(1)
\end{align*}

and 

\begin{align*}
&\int_1^l\ln(1+\frac{x(\mu-\beta)}{k})dx\\
&=\int_1^l\frac{x(\mu-\beta)}{k}dx-\frac{1}{2}\int_1^l \left(\frac{x(\mu-\beta}{k}\right)^2dx+\cdots\\
&=\frac{1}{2}\frac{\mu-\beta}{k}l^2+O(1).
\end{align*}

Hence 
\begin{align*}
a_{n_0+l}&\sim a_{n_0}e^{\frac{(\mu-\beta)\beta}{2\mu k}l^2-\frac{1}{2}\frac{\mu-\beta}{k}l^2}=a_{n_0}e^{-\frac{l^2}{\frac{2k\mu}{(\mu-\beta)^2}}}, 
\end{align*}
or, setting $\sigma^2 = \frac{k\mu}{(\mu-\beta)^2}$
$$a_{n_0+l}\sim a_{n_0}e^{-\frac{l^2}{2\sigma^2}}.$$

From Eq.\ref{defa}
\begin{align*}
a_{n_{0}}&=\left(\frac{\beta}{\mu}\right)^{n_{0}}\frac{\Gamma(\frac{k}{\beta})+n_{0}}{\Gamma(\frac{k}{\beta})n_{0}\,!}
\end{align*}
and, using Stirling's formula and the fact that $n_{0}\sim\frac{k}{\mu-\beta}$
\begin{align*} 
  a_{n_0} &= \left(\frac{\beta}{\mu}\right)^{n_0}\frac{\sqrt{\frac{2\pi}{\frac{k}{\beta}+n_{0}}}}{\sqrt{\frac{2\pi}{\frac{k}{\beta}}}}\frac{(\frac{\frac{k}{\beta}+n_{0}}{e})^{\frac{k}{\beta}+n_{0}}}{(\frac{\frac{k}{\beta}}{e})^{\frac{k}{\beta}}\sqrt{2\pi n_{0}}(n_{0}/e)^{n_{0}} }  \\
	&\sim \left(\frac{\beta}{\mu}\right)^{\frac{k}{\mu-\beta}}\frac{\sqrt{\frac{2\pi}{\frac{k}{\beta}+\frac{k}{\mu-\beta}}}}{\sqrt{\frac{2\pi}{\frac{k}{\beta}}}}\frac{(\frac{\frac{k}{\beta}+\frac{k}{\mu-\beta}}{e})^{\frac{k}{\beta}+\frac{k}{\mu-\beta}}}{\left(\frac{\frac{k}{\beta}}{e}\right)^{\frac{k}{\beta}}\sqrt{2\pi \frac{k}{\mu-\beta}}(\frac{k}{\mu-\beta}/e)^{\frac{k}{\mu-\beta}} }\\
             &= \frac{1}{\sqrt{2\pi}\sigma}\left(1+\frac{\beta}{\mu-\beta}\right)^{\frac{k}{\beta}}
\end{align*}
where $\sigma=\sqrt{\frac{\mu k}{(\mu-\beta)^2}}$.  Thus 
\begin{align*}
\frac{a_{n_{0}}}{\tilde S}&\sim \frac{\frac{1}{\sqrt{2\pi}\sigma}(1+\frac{\beta}{\mu-\beta})^{\frac{k}{\beta}}}{(1-\frac{\beta}{\mu})^{1- \frac{k}{\beta}}}\\&=\frac{1}{\sqrt{2\pi}\sigma} \left((1+\frac{\beta}{\mu-\beta})(1-\frac{\beta}{\mu})\right)^{\frac{k}{\beta}}\\
&=\frac{1}{\sqrt{2\pi}\sigma}\left(\frac{\mu}{\mu-\beta}\frac{\mu-\beta}{\mu}\right)^{\frac{k}{\beta}}\\
&=\frac{1}{\sqrt{2\pi}\sigma}
\end{align*}
and so 
$$\pi(n_{0}+l)=\frac{a_{n_0+l}}{\tilde{S}}\sim \frac{a_{n_0}}{\tilde{S}}e^{-\frac{l^2}{2\sigma^2}}\sim \frac{1}{\sqrt{2\pi}\sigma}e^{-\frac{l^2}{2\sigma^2}}\ \text{ as }\ n_{0}\rightarrow \infty.$$
\end{proof}


\subsection{Global Limit Theorems}
A functional Law of Large Numbers follows directly from Theorem 3.1 in Kurtz (1970 \cite{tk70}).  Likewise, a functional Central Limit Theorem follows from Theorems 3.1 and 3.5 in Kurtz (1971 \cite{tk71}).  We state these theorems here, therefore, without proof.

Write the population size as $n_k(t)$, a function of the immigration rate as well as time.  Set $n_k^* = \frac{k}{\mu - \beta}$, the limit of the first moment as $t \to \infty$.  Define a new stochastic process for the population size divided by the immigration rate, $Z_k(t) := \frac{n_k(t)}{k}$.  Set $z^* = \frac{n_k^*}{k} = \frac{1}{\mu - \beta}$. 

We define the transition function, $f_k(\frac{n_k}{k},j) := \frac{1}{k} p(n_k,n_k+j)$. Thus,
        \begin{equation*}
        f_{k}(z,j) = \left\{
        \begin{array}{l}
        \frac{\beta n_k + k}{k} = \beta z + 1 \hspace {1.4 cm} j=1 \\
        \frac{\mu n_k}{k} = \mu z \hspace {2.5 cm} j = -1\\
        \mathrm{(not\, needed)} \hspace {2 cm} j = 0 \\
        \end{array}
        \right.
        \end{equation*}  
Note that $f_k(z,j)$ does not, in fact, depend on $k$ and we write simply $f(z,j)$.

\begin{theorem}[Functional LLN]
Suppose $\lim \limits_{k\to \infty} Z_k(0) = z_0$.  Then, as $k \to \infty$, $Z_k(t) \to Z(t)$ uniformly in probability, where $Z(t)$ is a deterministic process, the solution of
	\begin{equation}\label{cm1}
	\frac{d Z(t)}{dt} = F(Z(t)), \ \ \ \ \ Z(0) = z_0. 
	\end{equation}
where 
	\begin{equation*}
   F(z) := \displaystyle\sum_{j} jf(z,j) = (\beta - \mu) z +1.
	\end{equation*}
\end{theorem}
This has the solution
\begin{equation*}
Z(t,z) = \frac{1}{\mu - \beta} + (z_0-\frac{1}{\mu - \beta}) e^{-(\mu - \beta)t} = z^* + (z_0-z^*) e^{-(\mu - \beta)t},\ t \ge 0.
\end{equation*}

Next, define $G_k(z) := \displaystyle\sum_{j}j^2 f_k(z,j) = (b + \mu)z + 1$.  This too does not depend on $k$ and we simply write $G(z)$.

\begin{theorem}[Functional CLT]
If $\lim \limits_{k\to \infty} \sqrt{k} \left(Z_k(0)-z^* \right) = \zeta_0$, the processes 
\[\zeta_k(t):=\sqrt{k}(Z_k(t) - Z(t)) \]
converge weakly in the space of cadlag functions on any finite time interval $[0, T]$ to a Gaussian diffusion $\zeta(t)$ with:
\begin{enumerate}
\item[1)] initial value $\zeta(0) = \zeta_0$, 

\item[2)] mean $$E\zeta(s) = \zeta_0 L_s := \zeta_0 e^{\int\limits_0^s F'(Z(u,z_0))du},$$

\item[3)] variance $$\mathrm{Var}(\zeta(s)) = L_s^2 \int\limits_0^s L_u^{-2} G(Z(u,z_0))du.$$
\end{enumerate}

Suppose, moreover, that $F(z_0) = 0$, i.e., $z_0 = z^*$, the equilibrium point. Then, $Z(t) \equiv z_0$ and $\zeta(t)$ is an Ornstein-Uhlenbeck process (OUP) with initial value $\zeta_0$, infinitesimal drift $$q := F'(z_0) = \beta - \mu$$ and infinitesimal variance $$a := G(z_0) = \frac{2\mu}{\mu - \beta}.$$

Thus, $\zeta(t)$ is normally distributed with mean $$\zeta_0 e^{qt} = \zeta_0 e^{-(\mu - \beta)t}$$ and variance $$\frac{a}{-2q}\left(1 - e^{2qt}\right) = \frac{\mu}{(\mu - \beta)^2}\left(1 - e^{-2(\mu - \beta)t}\right).$$
\end{theorem}
\bigskip



\section{Spatial Galton-Watson process with immigration \- and finite Markov chain spatial dynamics}
Let $X = \{x, y, \ldots \}$ be a finite set, and define the following parameters.
\begin{enumerate}
\item [] $\beta(x)$ is the rate of duplication at $x \in X$.
\item [] $\mu(x)$ is the rate of annihilation at $x \in X$.
\item [] $a(x,y)$ is the rate of transition $x \to y$.
\item [] $k(x)$ is	the rate of immigration into $x \in X$.
\end{enumerate}

We define $\overrightarrow{n}(t)=\{n(t,x), x\in X\}$, the population at moment $t \geq 0$, with $n(t,x)$ the occupation number of site $x \in X$.  Letting $\overrightarrow{\lambda}=\{\lambda_x \geq 0, x\in X\}$, we write the Laplace transform of the random vector $\overrightarrow{n}(t) \in \mathbb{R}^N$, $N = \mathrm{Card}(X)$ as $u(t, \overrightarrow{\lambda}) = E\, e^{-(\overrightarrow{\lambda}, \overrightarrow{n}(t))}$.

Now we derive the differential equation for $u(t, \overrightarrow{\lambda})$.  Denote the $\sigma$-algebra of events before or including $t$ by $\mathcal{F}_{\leq t}$.   Setting $\overrightarrow{\varepsilon}(t,dt)) = \overrightarrow{n}(t+dt) - \overrightarrow{n}(t)$
\begin{equation}\label{ltu1}
	u(t+dt, \overrightarrow{\lambda}) = E\, e^{-(\overrightarrow{\lambda},\overrightarrow{n}(t+dt))} = E\, e^{-(\overrightarrow{\lambda},\overrightarrow{n}(t))} E\, [e^{-(\overrightarrow{\lambda},\overrightarrow{\varepsilon}(t,dt))} |\mathcal{F}_{\leq t}] 
\end{equation}
The conditional distribution of $(\overrightarrow{\lambda},\overrightarrow{\varepsilon})$ under $\mathcal{F}_{\leq t}$ is given by the formulas
\begin{enumerate}
\item [a)] $P\{(\overrightarrow{\lambda},\overrightarrow{\varepsilon}(t,dt)) = \lambda_x |\mathcal{F}_{\leq t}\} = n(t,x)\beta(x)dt + k(x)dt$

(the birth of a new offspring at site $x$ or the immigration of a new particle into $x \in X$)
\item [b)] $P\{(\overrightarrow{\lambda},\overrightarrow{\varepsilon}) = \lambda_y |\mathcal{F}_{\leq t}\} = n(t,y)\mu(y)dt$

(the death of a particle at $y \in X$)
\item [c)] $P\{(\overrightarrow{\lambda},\overrightarrow{\varepsilon}) = \lambda_x - \lambda_z |\mathcal{F}_{\leq t}\} = n(t,x)a(x,z)dt;\ x,z\in X,\ x\ne z$

(transition of a single particle from $x$ to $z$.  Then, $n(t+dt,x) = n(t,x)-1$, $n(t+dt, z) = n(t,z) + 1$.)
\item [d)] $P\{(\overrightarrow{\lambda},\overrightarrow{\varepsilon}) = 0 |\mathcal{F}_{\leq t}\} = 1-(\displaystyle\sum_{x\in X} n(t,x) \beta(x))dt -(\displaystyle\sum_{x\in X} k(x)) dt - (\displaystyle\sum_{y\in X} n(t,y) \mu(y))dt - (\displaystyle\sum_{x\ne z} n(t,x) a(x,z))dt$
\end{enumerate}

After substitution of these expressions into \ref{ltu1} and elementary transformations we obtain
\begin{equation*}
\begin{split}
\frac{\partial u(t,\overrightarrow{\lambda})}{\partial t} = &E\, \displaystyle\sum_{x\in X} (e^{-\lambda x}-1)  e^{-(\overrightarrow{\lambda}, \overrightarrow{n}(t))} (\beta(x)n(t,x) + k(x)) +\\& \displaystyle\sum_{y\in X} (e^{\lambda y}-1) e^{-(\overrightarrow{\lambda}, \overrightarrow{n}(t))} \mu(y)n(t,y)+ \displaystyle\sum_{x,y; x\ne y} (e^{\lambda_x - \lambda y}-1) e^{-(\overrightarrow{\lambda}, \overrightarrow{n}(t))} a(x,y)n(t,x) \\
\end{split}
\end{equation*}

But 
\[E\, e^{-(\overrightarrow{\lambda}, \overrightarrow{n}(t))} n(t,x) = -\frac{\partial u(t,\overrightarrow{\lambda})}{\partial \lambda_x}\]
I.e., finally
\begin{equation}\label{ltu2}
\begin{split}
\frac{\partial u(t,\overrightarrow{\lambda})}{\partial t} = \displaystyle\sum_{x\in X} (e^{-\lambda_x}-1) & (-\frac{\partial u(t,\overrightarrow{\lambda})}{\partial \lambda_x} \beta(x) + u(t,\overrightarrow{\lambda}) k(x)) + \displaystyle\sum_{y\in X} (e^{\lambda_y}-1) \mu(y)(-\frac{\partial u(t,\overrightarrow{\lambda})}{\partial \lambda_y}) + \\
	&+ \displaystyle\sum_{x,z; x\ne z} (e^{\lambda_x - \lambda_z}-1) a(x,z)(-\frac{\partial u(t,\overrightarrow{\lambda})}{\partial \lambda_x}) \\
\end{split}
\end{equation}
The initial condition is
\[u(0,\overrightarrow{\lambda}) = E\, e^{-(\overrightarrow{\lambda},\overrightarrow{n}(0))} \]
(say, $u(0,\overrightarrow{\lambda}) = e^{-(\overrightarrow{\lambda},\mathbf{1})} = e^{\sum_{x\in X} \lambda_x}$ for $n(0,x) = 1$).

Differentiation of \ref{ltu2} and the substitution of $\overrightarrow{\lambda} = 0$ leads to the equations for the correlation functions (moments) of the field $n(t,x)$, $x \in X$.  Put
\[m_1(t,v)=E\, n(t,v)=-\frac{\partial u(t,\overrightarrow{\lambda})}{\partial \lambda_v}|_{\overrightarrow{\lambda}=0},\qquad v \in X \]
Then
\begin{equation*}
\begin{split}
\frac{\partial m_1(t,v)}{\partial t} &= k(v) + (\beta(v)-\mu(v)) m_1(t,v) + \frac{\partial}{\partial \lambda_v} \left(\sum_{z:z\ne v} (e^{\lambda_v-\lambda_z}-1) a(v,z) \frac{\partial u}{\partial \lambda_v} \right)|_{\overrightarrow{\lambda}=0} \\
	&\ \ + \frac{\partial}{\partial \lambda_v} \left(\sum_{z:z\ne v} (e^{\lambda_z-\lambda_v}-1) a(z,v) \frac{\partial u}{\partial \lambda_z} \right)|_{\overrightarrow{\lambda}=0} \\
	&= k(v) + (\underbrace{\beta(v)-\mu(v)}_{V(v)}) m_1(t,v) + \sum a(z,v)m_1(t,z) - (\sum_{z:z\ne v}a(v,z))m_1(t,v)
\end{split}
\end{equation*}
If $a(x,z) = a(z,x)$ then finally
\[\frac{\partial m_1(t,x)}{\partial t} = Am_1 + Vm_1 + k(x),\qquad m_1(0,x) = n(0,x) \]
Here, $A$ is the generator of a Markov chain $A=[a(x,y)] = A^*$.

By differentiating equation \ref{ltu2} over the variables $\lambda_x$, $x \in X$, one can get the equations for the correlation functions
\begin{equation*}
k_{l_1 \ldots l_m}(t,x_1,\ldots,x_m) = E\, n^{l_1}(t,x_1)\cdots n^{l_m}(t,x_m)
\end{equation*}
where $x_1, \ldots, x_m$ are different points of $X$ and $l_1, \dots, l_m \geq 1$ are integers. Of course
$k_{l_1 \ldots l_m}(t,x_1,\ldots,x_m) = (-1)^{l_1 + \dots + l_m}\frac{\partial^{l_1 + \dots + l_m} n(t,\overrightarrow{x})}{\partial^{l_1} \lambda_{x_1} \ldots \partial^{l_m} \lambda_{x_m}}|_{\overrightarrow{\lambda}=0}$.
The corresponding equations will be linear.  The central point here is that the factors $(e^{\lambda_x - \lambda_z}-1)$, $(e^{\lambda_y}-1)$, and $(e^{-\lambda_x}-1)$ are equal to 0 for $\overrightarrow{\lambda}=0$.  As a result, the higher order ($n>l_1 + \ldots + l_m$) correlation functions cannot appear in the equations for \{$k_{l_1 \ldots l_m}(\cdot)$, $l_1 + \ldots + l_m = n$\}.

Consider, for instance, the correlation function (in fact, matrix valued function)
\begin{equation*}
k_2(t,x_1,x_2) = \left[
\begin{array}{ll}
E\, n^2(t,x_1,x_1) & E\, n(t,x_1)\, n(t,x_2) \\
E\, n(t,x_1)\, n(t,x_2) & E\, n^2(t,x_2,x_2) \\
\end{array}\right]
\end{equation*}
\\
The method based on generating functions is typical for the theory of branching processes. In the case of processes with immigration, another, Markovian approach gives new results.  Let us start from the simplest case, when there is but one site, i.e., $X = \{x\}$.  Then, the process $n(t)$, $t \geq 0$ is a random walk with reflection on the half axis $n \geq 0$.

For a general random walk $y(t)$ on the half axis with reflection in continuous time, we have the following facts.  Let the process be given by the generator $G=(g(w,z))$, $w, z \geq 0$, where $a_w = g(w,w+1)$, $w\geq 0$; $b_w = g(w,w-1)$, $w > 0$; $g(w,w) = -(a_w+b_w)$, $w > 0$; and $g(0,0) = -a_0$ (see Fig. 1).

\begin{figure}[h]
\begin{center}
\includegraphics[width=\textwidth]{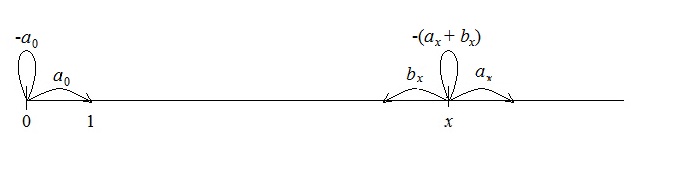}
\end{center}
\caption{General Random Walk with Reflection at 0}
\end{figure}
 
The random walk is recurrent iff the series 
\begin{equation}\label{crit1}
S = 1 + \frac{b_1}{a_1} + \ldots + \frac{b_1 \cdots b_n}{a_1 \cdots a_n} + \ldots 
\end{equation}
diverges.  It is ergodic (positively recurrent) iff the series
\begin{equation}\label{crit2}
\tilde{S} = 1 + \frac{a_0}{b_1} + \ldots + \frac{a_0 \cdots a_{n-1}}{b_1 \cdots b_n} + \ldots 
\end{equation}
converges.  In the ergodic case, the invariant distribution of the random walk $y(t)$ is given by the formula 
\begin{equation}\label{crit3}
\pi(n) = \frac{1}{\tilde{S}} \frac{a_0 \cdots a_{n-1}}{b_1 \cdots b_n} 
\end{equation}
(see \cite{fe1}).
\\

For our random walk, $n(t)$
\[\hspace{-8 cm} g(0,0)=-k,\qquad  a_0 = g(0,1) = k \]
and, for $n \geq 1$
\[b_n = g(n,n-1) = \mu n, \qquad g(n,n) = -(\mu n+\beta n + k), \qquad a_n = g(n,n+1) = \beta n + k. \]

\begin{Proposition}
\leavevmode
\begin{enumerate}
\item [1.] If $\beta > \mu$ the process $n(t)$ is transient and the population $n(t)$ grows exponentially.  

\item [2.] If $\beta = \mu$, $k > 0$ the process is not ergodic but rather it is zero-recurrent for $\frac{k}{\beta} \leq 1$ and transient for $\frac{k}{\beta} > 1$.

\item [3.] If $\beta < \mu$ the process $n(t)$ is ergodic.    The invariant distribution for $\beta<\mu$ is given by
\begin{equation*}
\begin{split}
\pi(n) &= \frac{1}{\tilde{S}}\frac{k(k+\beta)\cdots (k+\beta(n-1))}{\mu \cdot 2\mu \cdots n\mu} \\
	&= \frac{1}{\tilde{S}}\left(\frac{\beta}{\mu}\right)^n \frac{\frac{k}{\beta}(\frac{k}{\beta}+1)\cdots (\frac{k}{\beta}+n-1)}{n!} \\
	&= \frac{1}{\tilde{S}}\left(\frac{\beta}{\mu}\right)^n (1+\alpha)(1+\frac{\alpha}{2})\cdots (1+\frac{\alpha}{n}),\qquad \alpha = \frac{k}{\beta} - 1 \\
	&= \frac{1}{\tilde{S}}\left(\frac{\beta}{\mu}\right)^n \exp{\left(\sum_{j=1}^n\ln{(1+\frac{\alpha}{j})}\right)} \sim \frac{1}{\tilde{S}}\left(\frac{\beta}{\mu}\right)^n n^\alpha \\
\end{split}
\end{equation*}
where $\tilde{S} = \displaystyle\sum_{j=1}^\infty \frac{k(k+\beta)\cdots (k+\beta(j-1))}{\mu \cdot 2\mu \cdots j\mu}$.
\end{enumerate}
\end{Proposition}

\begin{proof}
1 and 3 follow from Eqs. \ref{crit1}, \ref{crit2}, and \ref{crit3}.  If $\beta = \mu$ (but $k > 0$), i.e., in the critical case, the process cannot be ergodic because, setting $\alpha = \frac{k}{\beta} -1$, then $\alpha > -1$ and as $n \to \infty$ $\tilde{S} \sim \displaystyle\sum_n n^\alpha = +\infty$.  The process is zero-recurrent, however, for $0<\frac{k}{\beta}\leq 1$.  In fact, for $\beta = \mu$
\begin{equation*}
\frac{b_1 \cdots b_n}{a_1 \cdots a_n} = \frac{\beta \cdot 2\beta \cdots n\beta}{(k+\beta) \cdots (k+n\beta)} = \frac{1}{\displaystyle\prod_{i=1}^n(1+\frac{k}{i\beta})} \asymp \frac{1}{n^{k/\beta}} \\
\end{equation*}
and the series in Eq. \ref{crit2} diverges if $0 < \frac{k}{\beta} \leq 1$.  If, however, $k > \beta$ the the series converges and the process $n(t)$ is transient.
\end{proof}

Consider, now, the general case of the finite space $X$.  Let $N = \mathrm{Card}\, X$ and $\overrightarrow{n}(t)$ be the vector of the occupation numbers.  The process $\overrightarrow{n}(t)$, $t \geq 0$ is the random walk on $(\mathbb{Z}_+^1)^N = \{0,1,...)^N$ with continuous time.  The generator of this random walk was already described when we calculated the Laplace transform $u(t, \overrightarrow{\lambda}) = E\, e^{-(\overrightarrow{\lambda}, \overrightarrow{n}(t))}$.  If at the moment $t$ we have the configuration $\overrightarrow{n}(t) = \{n(t,x), x\in X\}$, then, for the interval $(t, t+dt)$ only the following events (up to terms of order$(dt)^2$) can happen: 
\begin{enumerate}
\item [a)] the birth of a new particle at the site $x_0 \in X$, \- with corresponding probability $n(t,x_0) \beta(x_0)dt + k(x_0) dt$.  In this case we have the transition
\begin{equation*}
\overrightarrow{n}(t) = \{n(t,x), x \in X\} \to \overrightarrow{n}(t+dt) = \left\{
\begin{array}{l}
 n(t,x), x \ne x_0 \\
 n(t,x_0)+1, x=x_0 \\
\end{array}
\right.
\end{equation*}
\item [b)] the death of one particle at the site $x_0 \in X$.  This has corresponding probability $\mu(x_0) n(t,x_0)dt$ and the 
transition 
\begin{equation*}
\overrightarrow{n}(t) = \{n(t,x), x \in X\} \to \overrightarrow{n}(t+dt) = \left\{
\begin{array}{l}
 n(t,x), x \ne x_0 \\
 n(t,x_0)-1, x=x_0 \\
\end{array}
\right.
\end{equation*}
(Of course, here $n(t,x_0) \geq 1$, otherwise $\mu(x_0) n(t,x_0) dt = 0$).
\item [c)] 
the transfer of one particle from site $x_0$ to site $y_0 \in X$ (jump from $x_0$ to $y_0$), i.e., the 
transition 
\begin{equation*}
\overrightarrow{n}(t) = \{n(t,x), x \in X\} \to \overrightarrow{n}(t+dt) = \left\{
\begin{array}{l}
 n(t,x), x \ne x_0, y_0 \\
 n(t,x_0)-1, x=x_0 \\
 n(t,y_0)+1, x=y_0 \\
\end{array}
\right.
\end{equation*}
with probability $n(t,x_0)a(x_0,y_0) dt$ for $n(t,x_0) \geq 1$.
\end{enumerate}

The following theorem gives sufficient conditions for the ergodicity of the process $\overrightarrow{n}(t)$.
\begin{theorem}
Assume that for some constants $\delta > 0$, $A>0$ and any $x \in X$
\[\mu(x) - \beta(x) \geq \delta,\ k(x) \leq A. \]
Then, the process $\overrightarrow{n}(t)$ is an ergodic Markov chain and the invariant measure of this process has exponential moments, i.e., $E\, e^{(\overrightarrow{\lambda},\overrightarrow{n}(t))} \leq c_0 <\infty$ if $|\overrightarrow{\lambda}| \leq \lambda_0$ for appropriate (small) $\lambda_0 > 0$.
\end{theorem}

\begin{proof}
We take on $(\mathbb{Z}_+^1)^N = \{0,1,...)^N$ as a Lyapunov function
\[f(\overrightarrow{n}) = (\overrightarrow{n},\overrightarrow{\mathbf{1}}) = \displaystyle\sum_{x \in X} n_x,\qquad \overrightarrow{n} \in (\mathbb{Z}_+^1)^N, \]
Then, with $G$ the generator of the process, $G f(\overrightarrow{n}) \leq 0$ for large enough $(\overrightarrow{n},\overrightarrow{\mathbf{1}}) = \|\overrightarrow{n}\|_1$.  In fact
\[Gf = \displaystyle\sum_{x \in X} ((\beta(x)-\mu(x)) n_x + k(x)) < 0,\qquad \mathrm{for\, large\, }\|\overrightarrow{n}\|_1. \]
(The terms concerning transitions of the particles between sites make no contribution: $1-1=0$.)
\end{proof}

\begin{figure}[h]
\begin{center}
\includegraphics[width=\textwidth]{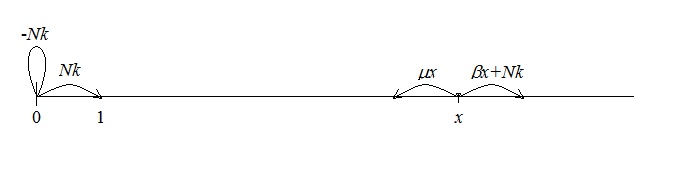}
\end{center}
\caption{Markov Model for Immigration Process}
\end{figure}

If $\beta(x) \equiv \beta < \mu \equiv \mu(x)$ and $k(x) \equiv k$ then $(\overrightarrow{n},\overrightarrow{\mathbf{1}})$, i.e., the total number of the particles in the phase space $X$ is also a Galton Watson process with immigration and the rates of transition shown in Fig. 2.

If $t \to \infty$ this process has a limiting distribution with invariant measure (in which $Nk$ replaces $k$).  That is
\[E\, (\overrightarrow{n},\overrightarrow{\mathbf{1}}) \xrightarrow[t \to \infty]{} \frac{Nk}{\mu - \beta} \]


\section{Branching process with migration and immigration}
We now consider our process with birth, death, migration, and immigration on a countable space, specifically the lattice $\mathbb{Z}^{d}$.  As in the other models, we have $\beta>0$, the rate of duplication at $x\in Z^{d}$; $\mu>0$, the rate of death; and $k>0$, the rate of immigration.  Here, we add migration of the particles with rate $\kappa>0$ and probability kernel $a(z)$, $z \in \mathbb{Z}^{d}$, $z\neq 0$, $a(z) = a(-z)$, $\sum\limits_{z\neq0}a(z)=1$.  That is, a particle jumps from site $x$ to $x+z$ with probability $\kappa a(z)dt$. Here we put $\kappa=1$ to simplify the notation. 

For $n(t,x)$ the number of particles at $x$ at time $t$, the forward equation for this process is given by 
$n(t+dt,x)=n(t,x)+\xi(dt,x)$, where
\begin{equation}
\xi(dt,x)=\left\{%
             \begin{array}{cl}
1 & \text{w. pr.}\,\,\,n(t,x)\beta dt+kdt+\sum\limits_{z\neq 0}a(z)n(t,x+z)dt\\
-1 & \text{w. pr.}\,\,\,n(t,x)(\mu + 1) dt \\
0 & \text{w. pr.}\,\,\,1-(\beta+\mu + 1) n(t,x)dt- \sum \limits_{z\neq 0}a(z)n(t,x+z)dt - kdt
\end{array}\right.
\end{equation}

Note that $\xi(dt,x)$ is independent on $\mathcal{F}_{\leqslant t}$ (the $\sigma$-algebra of events before or including $t$) and 
\begin{enumerate}
\item[a)] $E[\xi(dt,x)|\mathcal{F}_{\leqslant t}]=n(t,x)(\beta-\mu-1)dt+kdt+\sum \limits_{z\neq 0}a(z)n(t,x+z)dt$.
\item [b)] $E[\xi^{2}(dt,x)|\mathcal{F}_{\leqslant t}]=n(t,x)(\beta+\mu+1)dt+kdt+\sum \limits_{z\neq 0}a(z)n(t,x+z)dt$.
\item [c)] $E[\xi(dt,x)\xi(dt,y)|\mathcal{F}_{\leqslant t}]=a(x-y)n(t,x)dt+a(y-x)n(t,y)dt$.

A single particle jumps from $x$ to $y$ or from $y$ to $x$. Other possibilities have probability $O((dt)^2)\approx 0$. Here, of course, $x\neq y$.
\item [d)] If $x\neq y$, $y\neq z$, and $x\neq z$, then $E[\xi(dt,x)\xi(dt,y)\xi(dt,z)]=0$.

We will not use property d) in this paper but it is crucial for the analysis of moments of order greater or equal to $3$. \end{enumerate}
From here on, we concentrate on the first two moments.


\subsection{First moment}
Due to the fact that $\beta< \mu$, the system has a short memory, and we can calculate all the moments under the condition that $n(0,x)$, $x\in \mathbb{Z}^{d}$, is a system of independent and identically distributed random variables with expectation $\frac{k}{\mu-\beta}$. We will select Poissonian random variables with parameter $\lambda=\frac{k}{\mu-\beta}$. Then, $m_{1}(t,x)=\frac{k}{\mu-\beta}$, $t\geqslant 0$, $x \in \mathbb{Z}^{d}$, and, as a result, $\mathcal{L}_{a}m_{1}(t,x)=0$.
Setting $m_{1}(t,x)=E[n(t,x)]$, we have
\begin{equation}\label{m1eq}
\begin{aligned}
m_{1}(t+dt,x) &= E[E[n(t+dt,x)|\mathcal{F}_{t}]]= E[E[n(t,x)+\xi(t,x)|\mathcal{F}_{t}]]\\
&= m_{1}(t,x)+(\beta-\mu)m_{1}(t,x)dt+kdt+\sum\limits_{z\neq0}a(z)[m_{1}(t,x+z)-m_{1}(t,x)]dt
\end{aligned}
\end{equation}

Defining the operator $\mathcal{L}_{a}(f(t,x))=\sum\limits_{z\neq 0}a(z)[f(t,x+z)-f(t,x)]$, 
then, from Eq. \ref{m1eq} we get the differential equation 
\begin{equation*}
\left\{%
\begin{array}{cl}
\displaystyle\frac{\partial m_{1}(t,x)}{\partial t}&=(\beta-\mu)m_{1}(t,x)+k+\mathcal{L}_{a}m_{1}(t,x)\\
m_{1}(0,x)&=0
\end{array}\right.
\end{equation*}

Because of spatial homogeneity, $\mathcal{L}_{a}m_{1}(t,x)=0$, giving
\begin{equation*}
\left\{%
\begin{array}{cl}
\displaystyle\frac{\partial m_{1}(t,x)}{\partial t}&=(\beta-\mu)m_{1}(t,x)+k\\
m_{1}(0,x)&=0
\end{array}\right.
\end{equation*}
which has the solution
\begin{align*}
m_{1}(t,x)=\frac{k}{\beta-\mu}(e^{(\beta-\mu)t}-1).
\end{align*}
Thus, if $\beta\geq\mu$, $m_{1}(t,x)\rightarrow \infty$, and if $\mu>\beta$, $$\lim_{t\to \infty}m_{1}(t,x)=\frac{k}{\mu-\beta}.$$

\subsection{Second moment}
We derive differential equations for the second correlation function $m_2(t,x,y)$ for $x = y$ and $x \ne y$ separately, then combine them and use a Fourier transform to prove a useful result concerning the covariance.
\begin{enumerate}
\item[I.] $x=y$
\begin{equation*}
\begin{split}
m_{2}(t&+dt,x,x)=E[E[(n(t,x)+\xi(dt,x))^2|\mathcal{F}_{\leqslant t}]]\\
&=m_{2}(t,x,x)+2E[n(t,x)[n(t,x)(\beta-\mu-1)dt+kdt+\sum\limits_{z\neq0}a(z)n(t,x+z)]dt]\\
&\hspace{2 cm} +  E[n(t,x)(\beta+\mu+1)dt + kdt + \sum\limits_{z\neq 0}a(z)n(t,x+z)dt]\\
\end{split}
\end{equation*}
Denote $\mathcal{L}_{ax}m_{2}(t,x,y)=\sum\limits_{z\neq0}a(z)(m_{2}(t,x+z,y)-m_{2}(t,x,y))$.\\
From this follows the differential equation
\begin{equation*}
\left\{%
\begin{array}{cl}
\displaystyle\frac{\partial m_{2}(t,x,x)}{\partial t}&=2(\beta-\mu)m_{2}(t,x,x)+2\mathcal{L}_{ax}m_{2}(t,x,x)+\frac{2k^2}{\mu-\beta}+\frac{2k(\mu+1)}{\mu-\beta}\\
m_{2}(0,x,x)&=0
\end{array}\right.
\end{equation*}

\item[II.] $x \neq y$\\
Because only one event can happen during $dt$
$$P\{\xi(dt,x)=1,\xi(dt,y)=1\}=P\{\xi(dt,x)=-1,\xi(dt,y)=-1\}=0,$$
while the probability that one particle jumps from $y$ to $x$ is $$P\{\xi(dt,x)=1,\xi(dt,y)=-1\}=a(x-y)n(t,y)dt,$$
and the probability that one particle jumps from $x$ to $y$ is $$P\{\xi(dt,x)=-1,\xi(dt,y)=1\}=a(y-x)n(t,x)dt.$$
Then, similar to above
\begin{align*}
&m_{2}(t+dt,x,y)=E[E[(n(t,x)+\xi(t,x))(n(t,y)+\xi(t,y))|\mathcal{F}_{\leqslant t}]]\\
&=m_{2}(t,x,y)+(\beta-\mu)m_{2}(t,x,y)dt+km_{1}(t,y)dt+  \sum\limits_{z\neq0}a(z)(m_2(t,x+z,y)-m_{2}(t,x,y))dt \\
& + (\beta - \mu) m_{2}(t,x,y)dt + km_{1}(t,x)dt+\sum\limits_{z\neq0}a(z)(m_2(t,x,y+z) - m_2(t,x,y))dt\\
& +a(x-y)m_{1}(t,y)dt+a(y-x)m_{1}(t,x)dt\\
&=m_{2}(t,x,y)+2(\beta-\mu)m_{2}(t,x,y)dt+k(m_{1}(t,y)+m_{1}(t,x))dt+ (\mathcal{L}_{ax} + \mathcal{L}_{ay}) m_{2}(t,x,y)dt\\
& + a(x-y)(m_{1}(t,x)+m_{1}(t,y))dt
\end{align*}
The resulting differential equation is
\begin{align}
\begin{split}
\displaystyle\frac{\partial m_{2}(t,x,y)}{\partial t}&=2(\beta-\mu)m_{2}(t,x,y)+(\mathcal{L}_{ax}+\mathcal{L}_{ay})m_{2}(t,x,y)+k(m_{1}(t,x)+m_{1}(t,y))\\
& \,\hspace{2 cm} +a(x-y)[m_{1}(t,x)+m_{1}(t,y)]
\end{split}
\end{align}

That is 
\begin{equation*}
\begin{split}
\displaystyle\frac{\partial m_{2}(t,x,y)}{\partial t}&=2(\beta-\mu)m_{2}(t,x,y)+(\mathcal{L}_{ax}+\mathcal{L}_{ay})m_{2}(t,x,y)+\frac{2k^2}{\mu-\beta}+2a(x-y)\frac{k}{\mu-\beta}
\end{split}
\end{equation*}
\end{enumerate}

Because, for fixed $t$, $n(t,x)$ is homogeneous in space, we can write $m_{2}(t,x,y)=m_{2}(t,x-y)=m_{2}(t,u)$.  Then, we can condense the two cases into a single differential equation 
\begin{equation*}
\left\{%
\begin{array}{cl}
\displaystyle\frac{\partial m_{2}(t,u)}{\partial t}&=2(\beta-\mu)m_{2}(t,u)+2\mathcal{L}_{au}m_{2}(t,u)+\frac{2k^2}{\mu-\beta}+2a(u)\frac{k}{\mu-\beta}+\delta_{0}(u)\frac{2k(\mu+1)}{\mu-\beta}\\
m_{2}(0,u)&=En^{2}(0,x)
\end{array}\right.
\end{equation*}
Here $u=x-y\neq 0$ and $a(0)=0$.
\\

We can partition $m_{2}(t,u)$ into $m_{2}(t,u) = m_{21}+m_{22}$, where the solution for $m_{21}$ depends on time but not position and the solution for $m_{22}$ depends on position but not time.  Thus, $\mathcal{L}_{au}m_{21}=0$ and $m_{21}$ corresponds to the source $\frac{2k^2}{\mu-\beta}$, which gives
\begin{equation*}
\displaystyle\frac{\partial m_{21}(t,u)}{\partial t}=2(\beta-\mu)m_{21}(t,u)+\frac{2k^2}{\mu-\beta}
\end{equation*}
As $t \rightarrow \infty$, $m_{21}  \rightarrow \bar{M_{2}}=m_{1}^{2}(t,x)=\frac{k^2}{(\mu-\beta)^2}$. 

For the second part, $m_{22}$ , $\frac{\partial m_{22}}{\partial t}=0$, i.e.
\begin{equation*}
\displaystyle\frac{\partial m_{22}(t,u)}{\partial t}=2(\beta-\mu)m_{22}(t,u)+2\mathcal{L}_{au}m_{22}(t,u)+2a(u)\frac{k}{\mu-\beta}+\delta_{0}(u)\frac{2k(\mu+1)}{\mu-\beta}=0
\end{equation*}
As $t\rightarrow \infty$, $m_{22}\rightarrow \tilde{M}_{2}$. $\tilde{M}_2$ is the limiting correlation function for the particle field $n(t,x)$, $t\rightarrow \infty$. It is the solution of the ``elliptic" problem 
\begin{equation*}
2\mathcal{L}_{au}\tilde{M_{2}}(u)-2(\mu-\beta)\tilde{M_{2}}(u)+\delta_{0}(u)\frac{2k(\mu+1)}{\mu-\beta}+2a(u)\frac{k}{\mu-\beta}=0
\end{equation*}

Applying the Fourier transform $\widehat{\tilde{M_{2}}}(\theta)=\sum\limits_{u\in Z^{d}}\tilde{M_{2}}(u)e^{i(\theta,u)}$, $\theta\in T^{d}=[-\pi,\pi]^{d}$,\\
we obtain
$$\displaystyle\widehat{\tilde{M_{2}}}(\theta)=\displaystyle\frac{\frac{k}{\mu-\beta}+\frac{k \hat{a}(\theta)}{\mu-\beta}}{(\mu-\beta)+(1-\hat{a}(\theta)}.$$\\
 We have proved the following result.
\begin{theorem}
If $t\rightarrow \infty$, then $\mathrm{Cov}(n(t,x),n(t,y))=E[n(t,x)n(t,y)]-E[n(t,x)]E[n(t,y)]\\=\-m_{2}(t,x,y)-m_{1}(t,x)m_{1}(t,y)$, tends to $\tilde{M_{2}}(x-y)=\tilde{M_{2}}(u)\in L^2(Z^d)$\\
 The Fourier transform of $\tilde{M_{2}}(\cdot)$ is equal to 
 \begin{equation*}
 \widehat{\tilde{M_{2}}}(\theta)=\frac{c_{1}+c_{2}\hat{a}(\theta)}{c_{3}+(1-\hat{a}(\theta))} \in C(T^{d})
 \end{equation*}
 where $c_{1}=\frac{k}{\mu-\beta}$, $c_{2}=\frac{k}{\mu-\beta}$, $c_{3}=\mu-\beta$
\end{theorem}
 
Let's compare our results with the corresponding results for the critical contact model \cite{fmw} (where $k=0$, $\mu = \beta$). In the last case, the limiting distribution for the field $n(t,x)$, $t\geqslant 0$, $x\in \mathbb{Z}^{d}$, exists if and only if the underlying random walk with generator $\mathcal{L}_{a}$ is transient. In the recurrent case, we have the phenomenon of clusterization. The limiting correlation function is always slowly decreasing (like the Green kernel of $\mathcal{L}_{a}$).
 
In the presence of immigration, the situation is much better: the limiting correlation function always exists and we believe that the same is true for all moments. The decay of $\tilde{M}_{2}(u)$ depends on the smoothness of $\hat{a}(\theta)$. Under minimal regularity conditions, correlations have the same order of decay as $a(z)$, $z\rightarrow \infty$. For instance, if $a(z)$ is finitely supported or exponentially decreasing, the correlation also has an exponential decay. If $a(z)$ has power decay, then the same is true for correlation $\tilde{M}_{2}(u)$, $u\rightarrow \infty$.

\section{Processes in a Random Environment}
The final four models involve a random environment.  Two are Galton-Watson models with immigration and lack a spatial component.  In the first, the parameters are random functions of the population size; in the second, they are random functions of a Markov chain on a finite space.  The last two models are spatial and feature immigration, migration, and, most importantly, a random environment in space, still stationary in time for the third but not stationary in time for the fourth.

\subsection{Galton-Watson processes with immigration in random environments}

\subsubsection{Galton-Watson process with immigration in random environment based on population size}
Assume that rates of mortality $\mu(\cdot)$, duplication $\beta(\cdot)$, and immigration $k(\cdot)$ are random functions of the volume of the population $x\geq 0$. Namely, the random vectors $(\mu,\beta,k)(x,\omega)$ are i.i.d on the underlying probability space $(\Omega_{e},\mathcal{F}_{e},P_{e})$ ($e$: environment).

The Galton-Watson Process is ergodic ($P_{e}$-a.s) if and only if the random series 
$$S=\sum\limits_{n=1}^{\infty}\frac{k(0)(\beta(1)+k(1))(2\beta(2)+k(2))\cdots ((n-1)\beta (n-1)+k(n-1))}{\mu(1) (2 \mu(2)) \cdots (n\mu(n))} < \infty,\ P_{e} \text{-a.s.}$$

\begin{theorem}
Assume that the random variables $\beta(x,\omega)$, $\mu(x,\omega)$, $k(x,\omega)$ are bounded from above and below by the positive constants $C^{\pm}$: $0<C^{-}\leq \beta(x,\omega)\leq C^{+}<\infty$. Then, the process $n(t,\omega_{e})$ is ergodic $P_{e}$-a.s. if and only if $\langle \ln\frac{\beta(x,\omega)}{\mu(x,\omega)}\rangle = \langle \ln\beta(\cdot)\rangle -\langle \ln(\mu(\cdot))\rangle <0$
\end{theorem}

\begin{proof}
It is sufficient to note that $$\frac{k(n-1,\omega)+(n-1)\beta(n-1,\omega)}{n\mu(x,\omega)}=\frac{\frac{k(n-1,\omega)-\beta(n-1,\omega))}{n}+\beta(n-1,\omega)}{\mu(n,\omega)}=e^{\ln\beta(n-1)- \ln\mu(n)+o(\frac{1}{n})}.$$
\end{proof}

It follows from the strong LLN that  the series diverges exponentially fast for $\langle \ln\beta(\cdot)\rangle - \langle \ln\mu(\cdot) \rangle >0$; it converges like a decreasing geometric progression for $\langle \ln \beta(\cdot)\rangle -\langle \ln\mu(\cdot) \rangle < 0$; and it is divergent if $\langle \ln \beta(\cdot) \rangle = \langle \ln \mu(\cdot) \rangle $.  It diverges even when $\beta(x,\omega_{e})=\mu(x,\omega_{e})$ due to the presence of $k^{-}\geq C^{-}>0.$

Note that $ES<\infty$ if and only if $\langle \frac{\lambda(x-1)}{\mu(x)}\rangle = \langle \lambda \rangle \langle \frac{1}{\mu}\rangle <\infty$, i.e., the fluctuations of $S$, even in the case of convergence, can be very high.


\subsubsection{Random non-stationary(time dependent) environment}
Assume that $k(t)$ and $\Delta=(\mu-\beta)(t)$ are stationary random processes on $(\Omega_{m},P_{m})$ and that $k(t)$ is independent of $\Delta$.  For a fixed environment, i.e., fixed $k(\cdot)$ and $\Delta (\cdot)$, the equation for the first moment takes the form
\begin{align*}
\frac{dm_{1}(t,\omega_{m})}{dt}=-\Delta(t,\omega_{m})m_{1}+k(t,\omega_{m})\\
m_{1}(0,\omega_{m})=m_{1}(0)
\end{align*}
Then $$m_{1}(t,\omega_{m})=m_{1}(0)e^{-\int_{0}^{t}\Delta(u,\omega_{m})du}+\int_{0}^{t}k(s,\omega_{m})e^{-\int_{s}^{t}\Delta(u,\omega_{m})du}ds$$
Assume that $\frac{1}{\delta}\geqslant \Delta(\cdot) \geqslant \delta>0$, $\frac{1}{\delta}\geqslant k(\cdot) \geqslant \delta>0$.  Then
\begin{equation*}
m_{1}(t,\omega_{m})=\int_{-\infty}^{t}k(s,\omega_{m})e^{-\int_{s}^{t}\Delta(u,\omega_{m})du}ds+O(e^{-\delta t}).
\end{equation*}
Thus, for large $t$, the process $m_{1}(t,\omega_{m})$ is exponentially close to the stationary process
\begin{equation*}
\tilde m_{1}(t,\omega)=\int_{\infty}^{t}k(s,\omega_{m})e^{-\int_{s}^{t}\Delta(u,\omega_{m})du}ds
\end{equation*}

Assume now that $k(t)$ and $\Delta(s)$ are independent stationary processes and $-\Delta(t)=V(x(t))$, where $x(t)$, $t\geqslant 0$, is a Markov Chain with continuous time and symmetric geometry on the finite set $X$.  (One can also consider $x(t)$, $t \geqslant 0$, as a diffusion process on a compact Riemannian manifold with Laplace-Beltrami generator $\Delta$.)
Let 
\begin{align*}
u(t,x)&=E_{x}e^{\int_{0}^{t}V(x_{s})dx}f(x_{t})\\
&=E_{x}e^{\int_{0}^{t}-\Delta(x_{s})dx}f(x_{t})
\end{align*}
Then 
\begin{equation}\label{de52}
\left\{%
\begin{array}{cl}
&\displaystyle\frac{\partial u}{\partial t}=\mathcal{L}u+Vu=Hu\\\\
&u(0,x)=f(x)
\end{array}\right.
\end{equation}

The operator $\mathcal{L}$ is symmetric in $L^2(x)$ with dot product $(f,g)=\sum\limits_{x \in X}f(x)\bar{g(x)}$.  Thus, $H=\mathcal{L}+V$ is also symmetric and has real spectrum $0>-\delta \geqslant \lambda_{0}>\lambda_{1}\geqslant \cdots$ with orthonormal eigenfunctions $\psi_{0}(x)>0$,$\psi_{1}(x)>0$, $\cdots$  Iinequality $\lambda_{0}\leqslant \delta<0$ follows from our assumption on $\Delta(\cdot)$.

The solution of equation \ref{de52} is given by $$u(t,x)=\sum\limits_{n=1}^{N} e^{\lambda_{k}t}\psi_{k}(x)(t,\psi_{k}).$$

Now, we can calculate $<\tilde{m_{1}}(t,x,\omega_{m})>$.
\begin{equation}
<\tilde{m}>=\int_{-\infty}^{t}<k(\cdot)><E_{\pi}e^{\int_{s}^{t}V(x_{u})du}>ds
\end{equation}
Here, $\pi(x)=\frac{1}{N}=\frac{\mathbbm{1}(x)}{N}$ is the invariant distribution of $x_{s}$. Then
\begin{align*}
<\tilde{m}>&=\int_{-\infty}^{t}<k>\sum\limits_{k=0}^{k=N}e^{\lambda_{k}(t-s)}(\psi_{k}\pi)(\mathbbm{1}\psi_{k})ds\\
&=-<k>\sum\limits_{k=0}^{k=N}\frac{1}{\lambda_k}(\psi_{k}\mathbbm{1})^2\frac{1}{N}\\
&=-\frac{<k>}{N}\sum\limits_{k=0}^{N}\frac{(\psi_{k}\mathbbm{1})^2}{\lambda_{k}}
\end{align*}

\subsubsection{Galton-Watson process with immigration in random environment given by Markov chain}

  Let $x(t)$ be an ergodic MCh on the finite space $X$ and let $\beta(x),\mu(x),k(x)$, the rates of duplication, annihilation, and immigration, be functions from $X$ to $R^{+}$, and, therefore, functions of $t$ and $\omega_{e}$.  The process $(n(t),x(t))$ is a Markov chain on $\mathbb{Z}_{+}^{1}\times X$.
  
  Let $a(x,y)$, $x,y \in X$, $a(x,y) \geq 0$, $\sum\limits_{y \in X} a(x,y) = 1$ for all $x \in X$,  be the transition function for $x(t)$.  Consider $E_{(n,x)}f(n(t),x(t))= u(t,(n,x))$. Then 
\begin{equation*}
\begin{split}
u(t+dt,(n,x)) &= (1-(n\beta(x)+ n\mu(x) +  k(x) - a(x,x))dt)u(t,x) +n\beta(x)u(t,(n+1,x))dt \\
	& +k(x)u(t,(n+1,x))dt+n\mu(x)u(t,(n-1,x))dt+\sum\limits_{y:y\neq x}a(x,y) u(t,(n,y)) dt
\end{split}
\end{equation*}
We obtain the backward Kolmogorov equation
\begin{align*}
&\frac{\partial u}{\partial t}= \sum\limits_{y:y\neq z}a(t,y)(u(t,(n,y))-u(t,(n,x)))+(n\beta(x)+k(x))(u(t,(n+1,x))-u(t,(n,x))) \\
	& \hspace{2 cm} +n\mu(x)(u(t,(n-1,x))-u(t,(n,x))) \\
&u(0,(n,x))=0 
\end{align*}
\bigskip

\noindent \emph{Example.  Two-state random environment.}

\noindent Here, $x(t)$ indicates which one of two possible states, $\{1,2\}$ the process is in at time $t$.  The birth, mortality, and immigration rates are different for each state: $\beta_1$ and $\beta_2$, $\mu_1$ and $\mu_2$, and $k_1$ and $k_2$.  For a process in state 1, at any time the rate of switching to state 2 is $\alpha_1$, with $\alpha_2$ the rate of the reverse switch.  This creates the two-state random environment.  Let $G$ be the generator for the process, as diagrammed in Figure 3.

\begin{figure}[h]
\begin{center}
\includegraphics[width=\textwidth]{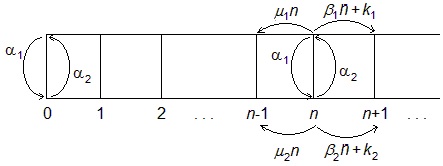}
\end{center}
\caption{GW process with immigration with random environment as two states}
\end{figure}

The following theorem gives sufficient conditions for the ergodicity of the process \- $(n(t),x(t))$.
\begin{theorem}
Assume that for some constants $\delta > 0$ and $A>0$ 
\[\mu_i - \beta_i \geq \delta,\ k_i \leq A, \ \ i=1,2 \]
Then, the process $(n(t),x(t))$ is an ergodic Markov chain and the invariant measure of this process has exponential moments, i.e., $E\, e^{\lambda n(t)} \leq c_0 <\infty$ if $\lambda \leq \lambda_0$ for appropriate (small) $\lambda_0 > 0$.
\end{theorem}

\begin{proof}
We take as a Lyapunov function $f(n,x) = n$.  

Then, $G f(n(t),x(t)) = (\beta_x - \mu_x)n + k_x$.
So for sufficiently large $n$, specifically $n > \frac{A}{\delta}$, we have $Gf \leq 0$.
\end{proof}


\subsection{Models with immigration and migration in a random environment}
For this most general case, we have migration and a non-stationary environment in space and time.  The rates of duplication, mortality, and immigration at time $t$ and position $x \in \mathbb{Z}^d$ are given by $\beta(t,x)$, $\mu(t,x)$, and $k(t,x)$.  As in the above models, immigration is uninfluenced by the presence of other particles; also set $\delta_{1}\leq k(t,x)\leq \delta_{2}$, $0 < \delta_{1} < \delta_{2} < \infty$.  The rate of migration is given by $\kappa$, with the process governed by the probability kernel $a(z)$,  the rate of transition from $x$ to $x+z$, $z\in Z^{d}$.

If $n(t,x)$ is the number of particles at $x \in Z^{d}$ at time $t$, $n(t+dt,x)=n(t,x)+\xi(t,x)$, where
\begin{equation*}
\xi(t,x)=\left\{%
             \begin{array}{cl}
1 & \textrm{w. pr.}\,\,\,n(t,x)\beta(t,x)dt+k(t,x)dt+\sum\limits_{z\neq 0} a(-z)n(t,x+z)dt\\
-1 & \text{w. pr.}\,\,\,n(t,x)\mu(t,x)dt+\sum\limits_{z\neq 0}a(z)n(t,x)dt\\
0 & \text{w. pr.}\,\,\, 1-(\beta(t,x)+\mu(t,x))n(t,x)dt-\sum\limits_{z\neq 0}a(z)n(t,x+z)dt \\ 
&\hspace{2 cm} -\sum\limits_{z\neq0}a(z)n(t,x)dt-k(t,x)dt
\end{array}\right.
\end{equation*}

For the first moment, $m_{1}(t,x)=E[n(t,x)]$, we can write
\begin{equation*}
\begin{aligned}
m_{1}(t+dt,x) &= E[E[n(t+dt,x)|\mathcal{F}_{t}]]= E[E[n(t,x)+\epsilon(t,x)|\mathcal{F}_{t}]]\\
&= m_{1}(t,x)+(\beta(t,x)-\mu(t,x))m_{1}(t,x)dt+k(t,x)dt \\
& \hspace{2 cm} +\sum\limits_{z\neq0}a(z)[m_{1}(t,x+z)-m_{1}(t,x)]dt
\end{aligned}
\end{equation*}
and so, defining, as above, $\mathcal{L}_{a}(f(t,x))=\sum\limits_{z\neq 0}a(z)[f(t,x+z)-f(t,x)]$, we obtain
\begin{equation}\label{differential equations for m1 }
\left\{%
\begin{array}{cl}
\displaystyle\frac{\partial m_{1}(t,x)}{\partial t}&=(\beta(t,x)-\mu(t,x))m_{1}(t,x)+k(t,x)+\mathcal{L}_{a}m_{1}(t,x)\\
m_{1}(0,x)&=0
\end{array}\right.
\end{equation}
\\

We consider two cases.  The first is where the duplication and mortality rates are equal, $\beta(t,x)=\mu(t,x)$.  Because of the immigration rate bounded above 0, we find that the expected population size at each site tends to infinity.  In the second case, to simplify, we consider $\beta(t,x)$ and $\mu(t,x)$ to be stationary in time, and assume the mortality rate to be greater than the duplication rate everywhere by at least a minimal amount.  Here, we show that the interplay between the excess mortality and the positive immigration results in a finite positive expected population size at each site.


\subsubsection{Case I}
If $\beta(t,x)=\mu(t,x)$
\begin{equation*}
\left\{%
\begin{array}{cl}
\displaystyle\frac{\partial m_{1}(t,x)}{\partial t}&=k(t,x)+\mathcal{L}_{a}m_{1}(t,x)\\
m_{1}(0,x)&=0
\end{array}\right.
\end{equation*}
Taking Fourier transforms,
\begin{equation*}
\left\{%
\begin{array}{cl}
\displaystyle\frac{\partial \widehat{ m_{1}}(t,v)}{\partial t}&=\widehat{k}(t,v)+\widehat{\mathcal{L}_{a}}(v)\widehat{m_{1}}(t,v)\\
\widehat{m_{1}}(0,x)&=0
\end{array}\right.
\end{equation*}

 \begin{align*}
\frac{\partial}{\partial t}(e^{-\widehat{\mathcal{L}_{a}}(v)t}\widehat{m_{1}} )
&=-\widehat{\mathcal{L}_{a}}(v)e^{ \widehat{\mathcal{L}_{a}}(v)t}\widehat{m_{1}}+e^{-\widehat{\mathcal{L}_{a}}(v)t}\frac{\partial \widehat{m_{1}}}{\partial t}=e^{ \widehat{\mathcal{L}_{a}}(v)t}\widehat{k}(t,v)
\end{align*}

 \begin{align*}
e^{-\kappa \widehat{\mathcal{L}_{a}}(v)t}\widehat{m_{1}}(t,v)
=\int_{0}^{t}e^{- \widehat{\mathcal{L}_{a}}(v)s}\widehat{k}(s,v)ds
\end{align*}

 \begin{align*}
\widehat{m_{1}}(t,v)
=\int_{0}^{t}e^{-(s-t) \widehat{\mathcal{L}_{a}}(v)}\widehat{k}(s,v)ds
\end{align*}
Taking the inverse Fourier transform,
 \begin{align*}
m_{1}(t,x)&=\frac{1}{(2\pi)^d}\int_{T_{d}}\int_{0}^{t}e^{-(s-t)\widehat{\mathcal{L}}_{a}(v)}\widehat{k}(s,v)dse^{-i(v,x)}dv\\
&=\int_{0}^{t}ds\sum\limits_{y\in Z^{d}}k(s,y)p(t-s,x-y,0)\geq\int_{0}^{t}\delta_{1}ds=\delta_{1} t
\end{align*}
where 
\begin{align*}
p(t,x,y)=\frac{1}{(2\pi)^{d}}\int_{T^{d}}e^{-t\widehat{\mathcal{L}}_{a}(v)-i(v,x-y)}dv=\frac{1}{(2\pi)^{d}}\int_{T^{d}}e^{-t\sum\limits_{j=1}^{d}(\cos{(v_{j})}-1)-i(v,x-y)}dv
\end{align*}

As $t \rightarrow \infty$, $\delta_{1}t\rightarrow \infty$.  Thus, when the birth rate equals the death rate, the expected population at each site $x\in \mathbb{Z}^{d}$ will go to infinity as $t\rightarrow\infty$.


\subsubsection{Case II}
Here, $\beta(t,x)\neq\mu(t,x)$.  For simplification we assume that only immigration, $k(t,x)$, is not stationary in time.  In other words, we us assume that the duplication and mortality rates \emph{are} stationary in time and depend only on position: $\beta(t,x)=\beta(x)$, $\mu(t,x)=\mu(x)$ and $\mu(x)-\beta(x)\geqslant \delta_{1}>0$.
From Eq. \ref{differential equations for m1 }, we get 
\begin{equation*}
\left\{%
\begin{array}{cl}
\displaystyle\frac{\partial m_{1}(t,x)}{\partial t}&=k(t,x)+\mathcal{L}_{a}m_{1}(t,x)+(\beta(t,x)-\mu(t,x))m_{1}(t,x)\\
m_{1}(0,x)&=0
\end{array}\right.
\end{equation*}

This has the solution 
\begin{equation*}
m_{1}(t,x)=\int_{0}^{t}ds\sum\limits_{y\in Z^{d}}k(s,y)q(t-s,x,y)
\end{equation*}
where $q(t-s,x,y)$ is the solution for
\begin{equation*}
\left\{
\begin{array}{cl}
\displaystyle\frac{\partial q}{\partial t}&=\mathcal{L}_{a}q+(\beta(t,x)-\mu(t,x))q\\
q(0,x,y)&=\delta(x-y)=\left\{
     \begin{array}{cl}
       1 & y=x\\
       0 & y\neq x
     \end{array}
   \right.
\end{array}\right.
\end{equation*}

By the Feynman-Kac formula,
\begin{equation*}
\begin{aligned}
q(s,x,y)&= E_{x}[e^{\int_{0}^{s}(\beta(x_{u})-\mu(x_{u}))du}\delta(x_{s}-y)]\\
&=E[e^{\int_{0}^{s}(\beta(x_{u})-\mu(x_{u})du}\delta(x_{s}-y))|x_{0}=x]\\
&=E[E[e^{\int_{0}^{s}(\beta(x_{u})-\mu(x_{u})du}\delta(x_{s}-y))|x_{0}=x,x_{s}=y]|x_{0}=x]\\
&=P(x_{s}=y|x_{0}=x)E_{x\rightarrow y}[e^{\int_{0}^{s}(\beta(x_{u})-\mu(x_{u})du}]\\
&=p(s,x,y)E_{x\rightarrow y}[e^{\int_{0}^{s}(\beta(x_{u})-\mu(x_{u})du}]
\end{aligned}
\end{equation*}
where 
\begin{equation*}
p(t,x,y)=\frac{1}{(2\pi)^d}\int_{T^d}e^{-t\widehat{\mathcal{L}}_{a}(v)-i(v,x-y)}dv.
\end{equation*}

Finally

\begin{align*}
\lim\limits_{t\rightarrow\infty}m_{1}(t,x)&=\lim\limits_{t\rightarrow\infty}\int_{0}^{t}ds\sum\limits_{y\in Z^{d}}k(s,y)E_{x\rightarrow y}[e^{\int_{0}^{t-s}(\beta(x_{u})-\mu(x_{u})du}]p(t-s,x,y)\\
& \hspace{2 cm} \textrm{and letting $w=t-s$}\\
&\leq\lim\limits_{t\rightarrow\infty}\int_{0}^{t}dw\lVert k \rVert_{\infty}E_{x\rightarrow y}[e^{\int_{0}^{w}(\beta(x_{u})-\mu(x_{u})du}] \\
&\leq \lVert k \rVert_{\infty}\int_{0}^{\infty}e^{-\delta_{1} w}dw\,\,\,\,\,\,\textrm{since $\beta(x)-\mu(x)\leq-\delta_{1}<0$}\\
&=\frac{\lVert k \rVert_{\infty}}{\delta_{1}}.
\end{align*}

Thus, when $\mu(x)-\beta(x)>0$, $\lim\limits_{t\rightarrow\infty}m_{1}(t,x)$ is bounded by $0$ and $\frac{\lVert k \rVert_{\infty}}{\delta_{1}}$, so this limit exists and is finite.



\end{document}